\documentclass{article}
\usepackage{amsmath,amssymb,fullpage}
\usepackage{graphicx}
\usepackage{tikz}
\usepackage{dsfont}
\usepackage{amsthm}
\usepackage[english]{babel}
\usepackage{fancyhdr}
\usepackage{wrapfig}
\usepackage{color}
\usepackage{enumitem}
\usepackage{hyperref}
\usepackage[utf8]{inputenc}
\usepackage{multirow}
\usetikzlibrary{quotes,angles,positioning}
\usetikzlibrary[arrows.meta,bending]

\newcommand{\Aut}[1]{\text{Aut}(#1)}
\newtheorem{theorem}{Theorem}[section]

\theoremstyle{definition}
\newtheorem{proposition}{Proposition}[section]

\newcommand{\myname}{Alejandra Brewer\footnote{Florida Southern College; breweralie@gmail.com}, 
Adam Gregory\footnote{Western Carolina University; gregoryadam9@gmail.com}, 
Quindel Jones\footnote{Jackson State University; quindel.d.jones@gmail.com}, 
Luke Rodriguez\footnote{University of Wisconsin; rodriglr@uw.edu}, \\ 
Rigoberto Fl\'{o}rez\footnote{The Citadel, rigoflorez@gmail.com}, 
and Darren Narayan\footnote{Rochester Institute of Technology, dansma@rit.edu}
}
\newcommand{\course}{Infinite Families of Asymmetric Graphs 
}
\title{\course}
\author{\myname}
\date{14 August 2018}
\begin{document}
\maketitle
\begin{abstract}
A graph $G$ is \textit{asymmetric} if its automorphism group of vertices is trivial. Asymmetric graphs were introduced by Erd\H{o}s and R\'{e}nyi in 1963. They showed that the probability of a graph on $n$ vertices being asymmetric tends to $1$ as $n$ tends to infinity. In this paper, we first give consider the number of asymmetric trees, a question posed by Erd\H{o}s and R\'enyi. We give a partial result, showing that the number of asymmetric subdivided stars is approximately $q(n-1) - \lfloor \frac{n-1}{2} \rfloor$ where $q(n)$ is the number of ways to sum to $n$ using distinct positive integers, found by Hardy and Ramanujan in 1918. We also investigate cubic Hamiltonian graphs where asymmetry, at least for small values of $n$, seems to be rare. It is known that none of the cubic Hamiltonian graphs on $4\leq n\leq 10$ vertices are asymmetric, and of the $80$ cubic Hamiltonian graphs on $12$ vertices only $5$ are asymmetric. We give a construction of an infinite family of cubic Hamiltonian graphs that are asymmetric. Then we present an infinite family of quartic Hamiltonian graphs that are asymmetric. We use both of the above results for cubic and quartic asymmetric Hamiltonian graphs to establish the existence of $k$-regular asymmetric Hamiltonian graphs for all $k\geq 3$. 

\indent
\end{abstract}

\section{Introduction}
	
\hspace{10pt} We consider undirected graphs without multiple edges or loops. A graph $G$ is \textit{asymmetric} if its automorphism group of vertices is trivial. Asymmetric graphs were introduced by Erd\H{o}s and R\'{e}nyi in 1963. They showed that the probability of a graph on $n$ vertices being asymmetric tends to $1$ as $n$ tends to infinity. In this paper, we investigate cubic Hamiltonian graphs where asymmetry, at least for small values of $n$, seems to be rare. It is known that none of the cubic Hamiltonian graphs on $4\leq n\leq 10$ vertices are asymmetric, and of the $80$ cubic Hamiltonian graphs on $12$ vertices only $5$ are asymmetric \cite{Yutsis}. 
	
We will use $n$ to denote the number of vertices in a graph. For a graph $G$ we will use $V(G)$ to denote the set of vertices in $G$, and $E(G)$ to denote the set of edges in $G$. The edge between vertices $u$ and $v$ will be denoted $uv$. Two graphs $G$ and $H$ are \textit{isomorphic} if there is a bijection $f:G \rightarrow H$ where $uv \in E(G) \Leftrightarrow f(u)f(v) \in E(H)$. Recall that $f$ is an \textit{automorphism} if it is an isomorphism from a graph to itself, and the set of all automorphisms of a graph form an algebraic group under function composition. We will use $\Aut{H}$ to denote the automorphism group of vertices in a graph $H$. The \textit{complement} of a graph $G$ will be denoted $\overline{G}$. We will use $C_{n}$ will denote a cycle on $n$ vertices; $K_{s,t}$ to denote a complete bipartite graph where one part has $s$ vertices and the other part has $t$ vertices; and $P_{n}$ will denote a path on $n$ vertices. The \textit{degree} of a vertex $v$ is the number of edges incident to $v$.  A graph is $k$\textit{-regular} if each vertex has degree $k$. An edge-subdivision of an edge $uv$ is performed by replacing the edge $uv$ by two edges $uw$ and $wv$, where $w$ is a new vertex. A subdivided star is a graph that can be obtained by subdividing one more more edges of the star $K_{1,n-1}$. The \textit{distance} between two vertices $u$ and $v$ is the number of edges in a shortest path between $u$ and $v$. The distance between an edge and a subgraph will be the minimum vertex distance between a vertex incident to the edge and a vertex in the subgraph. For any undefined notation, please see the text \cite{West} by West. 
	
    In this paper we first investigate asymmetric trees. We show that the number of asymmetric subdivided stars on $n$ vertices equals the number of partitions of an integer into at least three distinct parts, a problem studied by Hardy and Ramanujan \cite{Hardy} in 1918. Furthermore, we identify the smallest asymmetric tree that is not a subdivided star.
    
We also give a construction of an infinite family of cubic Hamiltonian graphs that are asymmetric. Then we present an infinite family of quartic Hamiltonian graphs that are asymmetric. We use both of the above results to establish the existence of $k$-regular asymmetric Hamiltonian graphs for all $k\geq 3$.

A vertex is \textit{unique} if all of its properties are different from the properties of all of the other vertices in the graph. A graph $G$ can be shown to have a trivial automorphism group if each vertex is unique. For any graph $G$, two edges $e_1,e_2 \in E(G)$ are said to be \textit{different} if and only if $G -\{e_1\} \ncong G - \{e_2\}$. It is known that if $G$ has no component isomorphic to $P_{2}$, then $\Aut{V(G)} \cong \Aut{E(G)}$.  In this way, one could alternatively show that each edge in $G$ is different. In the paper, we will show that two edges are different by showing that one edge has a property (e.g. included in a particular type of subgraph, or is incident to a vertex that is unique) that another edge does not. 

\begin{proposition}
Given any graph $G$, $\Aut{G} = \Aut{\overline{G}}$.
\end{proposition}
\vspace{-12pt}
\begin{proof}
Let $\sigma \in \Aut{G}$. Then $e \in E(\overline{G})\Rightarrow e\notin E(G)\Rightarrow \sigma (e)\notin E(G)\Rightarrow \sigma (e)\in E(\overline{G})$. This implies $\sigma \in \Aut{\overline{G}}$ and hence $\Aut{G} \subseteq \Aut{\overline{G}}$. Using a similar argument, we can show $\Aut{\overline{G}} \subseteq \Aut{G}$, which completes the proof.
\end{proof}
	
\begin{theorem}Every asymmetric graph on $n$ vertices can be extended to an asymmetric graph on $n+1$ vertices. 
\end{theorem}
\vspace{-12pt}
\begin{proof}
Let $G$ be an asymmetric graph. If $G$ has no vertex of degree one, then we can add a new vertex $w$ along with the edge $wv$, where $v$ is a vertex of largest degree. Note that all of the vertices that were unique in $G-v$ will remain unique. The vertex $v$ will be unique since it has the largest degree, and the vertex $w$ is unique since it has the smallest degree. \\ 
{\indent} If $G$ has a vertex of degree one, let $u$ be a vertex with degree one and the greatest distance from a vertex of degree greater than two. Then we can add a new vertex $x$ along with an edge $xu$. Note that all of the vertices that were unique in $G-u$ will remain unique. The vertex $u$ will be unique since it is the only vertex that is adjacent to a vertex of degree one, and vertex $x$ is unique since it is the only vertex of degree one.
\end{proof}

In Section 2, we give results for the enumeration of a class of asymmetric trees. In Section 3, we investigate asymmetric cubic Hamiltonian graphs and present an infinite family of these graphs. In Section 4, we investigate asymmetric quartic Hamiltonian graphs and again provide an infinite family.  Lastly, Section 5 summarizes the results of the paper and states open problems.

 \section{Trees and Subdivided Stars}
The smallest asymmetric tree was given by Erd\H{o}s and R\'{e}nyi in \cite{Erdos} and is depicted in Figure 1. 

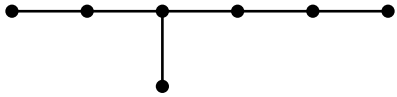
\begin{figure}[h!]
\begin{center}
\begin{tikzpicture}[node distance = 0.1cm, line width = 1pt]
\coordinate (1) at (0,0);
\coordinate (2) at (1,0);
\coordinate (3) at (2,0);
\coordinate (4) at (3,0);
\coordinate (5) at (4,0);
\coordinate (6) at (5,0);
\coordinate (7) at (2,-1);

\draw(1)--(2);
\draw(2)--(3);
\draw(3)--(4);
\draw(3)--(7);
\draw(4)--(5);
\draw(5)--(6);

\foreach \point in {1,2,3,4,5,6,7} \fill (\point) circle (2.5pt);
\end{tikzpicture}
\end{center}
\caption{The smallest asymmetric tree}
\end{figure} 

Notice that the smallest asymmetric tree is a subdivided star with $n=7$. We use $T_{n_1, n_2, \ldots, n_k}$ to denote a subdivided star with a vertex of degree $k$, which we refer to as the center vertex,  and pendant paths $P_{n_1}, P_{n_2}, \ldots,P_{n_k}$ extending from the center vertex.
Hence, the graph shown in Figure 1 will be denoted $T_{1,2,3}$. If we continue to extend any pendant path of the subdivided star on seven vertices, such that no two paths are the same length, we build an asymmetric subdivided star on $n$ vertices. Figure 2 depicts the asymmetric subdivided star $T_{1,2,4}$ when $n=8$. Notice that each pendant path has a distinct length. 

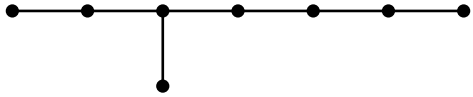
\begin{figure}[h!]
\begin{center}
\begin{tikzpicture}[node distance = 0.1cm, line width = 1pt]
\coordinate (1) at (0,0);
\coordinate (2) at (1,0);
\coordinate (3) at (2,0);
\coordinate (4) at (3,0);
\coordinate (5) at (4,0);
\coordinate (6) at (5,0);
\coordinate (7) at (2,-1);
\coordinate (8) at (6,0);

\draw(1)--(2);
\draw(2)--(3);
\draw(3)--(4);
\draw(3)--(7);
\draw(4)--(5);
\draw(5)--(6);
\draw(6)--(8);

\foreach \point in {1,2,3,4,5,6,7,8} \fill (\point) circle (2.5pt);
\end{tikzpicture}
\end{center}
\caption{The only asymmetric tree on eight vertices}
\end{figure} 

\begin{theorem}
If $G = T_{n_1, n_2, \ldots, n_k}$ where $k \geq 3$ and $n_1, n_2, \ldots, n_k$ are all unique, then $G$ is asymmetric. 
\end{theorem}

\begin{proof}
Let $G =  T_{n_1, n_2, \ldots, n_k}$ where $k \geq 3$ and $n_1, n_2, \ldots, n_k$ are all distinct. Then any two vertices on the same pendant path have a different distance to the center vertex $v$, which is unique as it has degree $k$. In this way, any two vertices on the same pendant path are unique. For vertices on different pendant paths, if we consider $G-v$ we have $k$ non-isomorphic paths. Thus, any two vertices on different pendant paths are in components of different sizes, and are therefore unique.
\end{proof}

We count the number of asymmetric subdivided stars using integer partitions. A partition of $n-1$ into three or more distinct parts ensures a sub-divided star where no two pendant paths have the same length. In \cite{Hardy}, Hardy and Ramanujan partition integers into distinct parts and give a formula for counting them. We adopt the notation $q(n)$ for this formula, as used in \cite{Hua}. However, because a partition on less than three parts yields a path, which is symmetric, we are only interested in distinct partitions of three or more parts. Moreover, for a subdivided star on $n$ vertices we consider the integer partitions of $n-1$. Thus, removing partitions of integers with only two parts from $q(n-1)$ results in our formula gives the number of asymmetric sub-divided stars,
$|ASDS_{n}|=q(n-1)-\left\lfloor \frac{n-1}{2}\right\rfloor$.

The smallest asymmetric tree has seven vertices.  This graph can be extended to the only asymmetric tree on eight vertices, $T_{1,2,4}$. This graph can be extended to two non-isomorphic asymmetric trees on nine vertices: $T_{1,2,5}$ and $T_{1,3,4}$. However, a vertex sub-division could be performed on the central vertex of $T_{1,2,4}$ to obtain the graph found in Figure 3. By Theorem 1.1, it is possible to extend the longest pendant path of the graph in Figure 3 to create an infinite family of asymmetric trees.

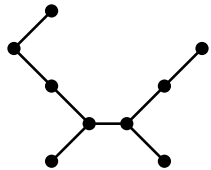
\begin{figure}[h!]
\begin{center}
\begin{tikzpicture}[node distance = 0.1cm, line width = 1pt]
\coordinate (1) at (0,0);
\coordinate (2) at (0.5,0);
\coordinate (3) at (1,0.5);
\coordinate (4) at (1.5,1);
\coordinate (5) at (1,-0.5);
\coordinate (6) at (-0.5,-0.5);
\coordinate (7) at (-0.5,0.5);
\coordinate (8) at (-1,1);
\coordinate (9) at (-0.5,1.5);

\draw(1)--(2);
\draw(2)--(3);
\draw(3)--(4);
\draw(2)--(5);
\draw(1)--(6);
\draw(1)--(7);
\draw(7)--(8);
\draw(8)--(9);

\foreach \point in {1,2,3,4,5,6,7,8,9} \fill (\point) circle (2.5pt);
\end{tikzpicture}
\end{center}
\caption{The smallest asymmetric tree that is not a subdivided star}

\end{figure}


We note by Theorem 1.1, every asymmetric tree on $n$ vertices can be extended to an asymmetric tree on $n+1$ vertices. Hence all asymmetric trees can be extended to infinite families.

\section{Asymmetric Cubic Hamiltonian Graphs}
In this section we provide a construction for an infinite family of asymmetric cubic Hamiltonian graphs.
\vspace{-10pt}
\subsection{Construction}
We begin by detailing a procedure for constructing the graph:
\begin{itemize}
\item[(1)] Construct $C_n$ where $n$ is even and $n \geq 12$ and label the vertices from $v_1$ to $v_n$, clockwise.
\vspace{-5pt}
\item[(2)] Add the edge $v_n v_{\frac{n}{2}-1}$ which creates two unequal sets of vertices.
\vspace{-5pt}
\item[(3)] Add the edge $v_1 v_{n-2}$. 
\vspace{-5pt}
\item[(4)] Add the edge $v_{n-1} v_{n-3}$.
\vspace{-5pt}
\item[(5)] For each $v_k \in G$, if $2 \leq k  \leq \frac{n}{2}-2$, then add the edge $v_k v_{n-2-k}$. 
\end{itemize}

This yields a cubic Hamiltonian graph on $n$ vertices, and examples of this graph on 12 and 18 vertices are depicted in Figure 4. Note that there are exactly two $C_3$ subgraphs, which remain fixed as $n$ increases. We will now show this graph has no symmetry.

\begin{figure}[h!]
\begin{center}
\begin{tikzpicture}[node distance = 0.1cm, line width = 1pt]
\coordinate (1) at ({2.5*cos(30)},{2.5*sin(30)});
\coordinate (2) at ({2.5*cos(60)},{2.5*sin(60)});
\coordinate (3) at ({2.5*cos(90)},{2.5*sin(90)});
\coordinate (4) at ({2.5*cos(120)},{2.5*sin(120)});
\coordinate (5) at ({2.5*cos(150)},{2.5*sin(150)});
\coordinate (6) at ({2.5*cos(180)},{2.5*sin(180)});
\coordinate (7) at ({2.5*cos(210)},{2.5*sin(210)});
\coordinate (8) at ({2.5*cos(240)},{2.5*sin(240)});
\coordinate (9) at ({2.5*cos(270)},{2.5*sin(270)});
\coordinate (10) at ({2.5*cos(300)},{2.5*sin(300)});
\coordinate (11) at ({2.5*cos(330)},{2.5*sin(330)});
\coordinate (12) at ({2.5*cos(0)},{2.5*sin(0)});
\node (13) [above right = 0.25mm of 1] {2};
\node (14) [above right = 0.25mm of 2] {1};
\node (15) [above = 0.5mm of 3] {12};
\node (16) [above left = 0.25mm of 4] {11};
\node (17) [above left = 0.25mm of 5] {10};
\node (18) [left = 0.5mm of 6] {9};
\node (19) [below left = 0.25mm of 7] {8};
\node (20) [below left = 0.25mm of 8] {7};
\node (21) [below = 0.5mm of 9] {6};
\node (22) [below right = 0.25mm of 10] {5};
\node (23) [below right = 0.25mm of 11] {4};
\node (24) [right = 0.5mm of 12] {3};

\coordinate (25) at ({6.5 + 2.5*cos(20)},{2.5*sin(20)});
\coordinate (26) at ({6.5 + 2.5*cos(40)},{2.5*sin(40)});
\coordinate (27) at ({6.5 + 2.5*cos(60)},{2.5*sin(60)});
\coordinate (28) at ({6.5 + 2.5*cos(80)},{2.5*sin(80)});
\coordinate (29) at ({6.5 + 2.5*cos(100)},{2.5*sin(100)});
\coordinate (30) at ({6.5 + 2.5*cos(120)},{2.5*sin(120)});
\coordinate (31) at ({6.5 + 2.5*cos(140)},{2.5*sin(140)});
\coordinate (32) at ({6.5 + 2.5*cos(160)},{2.5*sin(160)});
\coordinate (33) at ({6.5 + 2.5*cos(180)},{2.5*sin(180)});
\coordinate (34) at ({6.5 + 2.5*cos(200)},{2.5*sin(200)});
\coordinate (35) at ({6.5 + 2.5*cos(220)},{2.5*sin(220)});
\coordinate (36) at ({6.5 + 2.5*cos(240)},{2.5*sin(240)});
\coordinate (37) at ({6.5 + 2.5*cos(260)},{2.5*sin(260)});
\coordinate (38) at ({6.5 + 2.5*cos(280)},{2.5*sin(280)});
\coordinate (39) at ({6.5 + 2.5*cos(300)},{2.5*sin(300)});
\coordinate (40) at ({6.5 + 2.5*cos(320)},{2.5*sin(320)});
\coordinate (41) at ({6.5 + 2.5*cos(340)},{2.5*sin(340)});
\coordinate (42) at ({6.5 + 2.5*cos(0)},{2.5*sin(0)});
\node (43) [above = 0.5mm of 29] {18};
\node (44) [above = 0.5mm of 28] {1};
\node (45) [above right = 0.25mm of 27] {2};
\node (46) [above right = 0.25mm of 26] {3};
\node (47) [right = 0.5mm of 25] {4};
\node (48) [right = 0.5mm of 42] {5};
\node (49) [right = 0.5mm of 41] {6};
\node (50) [below right = 0.25mm of 40] {7};
\node (51) [below right = 0.25mm of 39] {8};
\node (52) [below = 0.5mm of 38] {9};
\node (53) [below = 0.5mm of 37] {10};
\node (54) [below left = 0.25mm of 36] {11};
\node (55) [below left = 0.25mm of 35] {12};
\node (56) [left = 0.5mm of 34] {13};
\node (57) [left = 0.5mm of 33] {14};
\node (58) [left = 0.5mm of 32] {15};
\node (59) [above left = 0.25mm of 31] {16};
\node (60) [above left = 0.25mm of 30] {17};
\node (61) [below = 0.75mm of 21] {(a)};

\draw \foreach \x [remember=\x as \lastx (initially 1)] in {2,3,4,5,6,7,8,9,10,11,12,1}{(\lastx) -- (\x)};
\draw \foreach \x [remember=\x as \lastx (initially 25)] in {26,27,28,29,30,31,32,33,34,35,36,37,38,39,40,41,42,25}{(\lastx) -- (\x)};
\draw (6) to [out=40, in=-100] (4);
\draw (9) to [out=50, in=190] (11);
\draw (3)--(10);
\draw (1)--(7);
\draw (12)--(8);
\draw (2)--(5);

\draw (29)--(39);
\draw (27)--(33);
\draw (26)--(34);
\draw (25)--(35);
\draw (42)--(36);
\draw (41)--(37);

\draw (31) to [out=5, in=-140] (28);
\draw (32) to [out=20, in=-100] (30);
\draw (38) to [out=50, in=190] (40);
\draw (37) to node [below=7mm]{(b)} (38);

\foreach \point in {1,2,3,4,5,6,7,8,9,10,11,12,25,26,27,28,29,30,31,32,33,34,35,36,37,38,39,40,41,42} \fill (\point) circle (2.5pt);

\end{tikzpicture} 
\end{center}
\caption{Asymmetric cubic Hamiltonian graphs}
\end{figure}
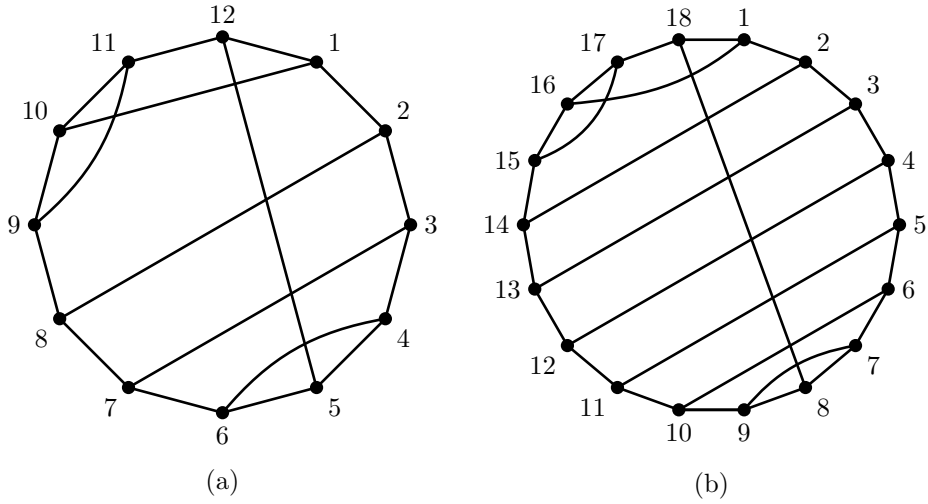

\subsection{Ensuring a Trivial Automorphism Group}

\hspace{10pt} In order to show that this cubic Hamiltonian graph on $n$ vertices has a trivial automorphism group, we will show that every edge is different from all other edges in the graph. We begin by partitioning edges into sets based on various properties, such that edges in one set are different from edges in all other sets. \\ 

\noindent We partition $E(G)$ as follows:
\begin{itemize}
\item $E(\text{I}) = \{v_1 v_2\}$\\ 
Edges incident to two $C_4$ subgraphs and are not contained in a $C_4$.
\vspace{-5pt}
\item $E(\text{II}) = \{v_n v_{n-1},  v_n v_1,  v_1 v_{n-2},  v_{2} v_{n-4}, v_{\frac{n}{2}-3} v_{\frac{n}{2}+1}\}$\\
Edges that are in a $C_4$ and in a $C_5$.
\vspace{-5pt}
\item $E(\text{III})= \{v_{n-2}v_{n-1}, v_{n-1}v_{n-3}, v_{\frac{n}{2}-1} v_{\frac{n}{2}}, v_{\frac{n}{2}-1}v_{\frac{n}{2}-2}\}$\\
Edges that are in a $C_3$ and incident to a vertex that is distance 2 from a vertex in the other $C_3$.
\vspace{-5pt}
\item $E(\text{IV}) = \{v_{n-3} v_{n-2}, v_{\frac{n}{2}-2}v_{\frac{n}{2}}\}$ \\ 
Edges that are in a $C_3$ but not incident to a vertex that is distance 2 from a vertex in the other $C_3$. 
\vspace{-5pt}
\item $E(\text{V})=\{v_nv_{\frac{n}{2}-1}, v_{n-3}v_{n-4}\}$\\
Edges that are incident to a $C_3$ (and are not in a $C_4$ nor a $C_5$).
\vspace{-5pt}
\item $E(\text{VI})=\{v_kv_{n-2-k} \; | \; 2 \leq k \leq \frac{n}{2}-4\}$ \\
Edges that are in two $C_4s$, but no $C_5s$.
\vspace{-5pt}
\item $E(\text{VII})=\{v_kv_{k+1} \; | \; 2 \leq k \leq \frac{n}{2}-3 \; \text{and} \; \frac{n}{2} \leq k \leq n-5 \}$ \\
Edges that are in a $C_4$ and incident to a $C_4$.
\end{itemize}

\hspace{10pt} In this way, edges in each set are different from edges in all other sets. Now we must also show that every edge in each set is different from all other edges within the same set.

\begin{itemize}
\item $E(\text{I})$:
\vspace{-10pt}
\begin{itemize}
\item This set has only one edge.
\end{itemize}
\vspace{-10pt}
\item $E(\text{II})$:
\vspace{-10pt}
\begin{itemize}
\item Edge $v_nv_{n-1}$ has three distinct paths of length one to a $C_3$ subgraph. 

\item Edge $v_nv_1$ is distance one from the $C_3$ subgraph with vertices $v_{\frac{n}{2}}, v_{\frac{n}{2}-1}, v_{\frac{n}{2}-2}$, while the edge $v_1v_{n-2}$ is not.
\end{itemize}
\noindent In this way, the edges $v_nv_{n-1},v_nv_1$ and $v_1v_{n-2}$ are different.

\begin{itemize}
\item The edges $v_2v_{n-4}$ and $v_{\frac{n}{2}-3}v_{\frac{n}{2}+1}$ are different from the edges $v_nv_{n+1}, v_nv_1$ and $v_1v_{n-2}$, as the latter edges are distance one or less from the $C_3$ with vertices $v_{n-1},v_{n-2}, v_{n-3}$, whereas the edges $v_2v_{n-4}$ and $v_{\frac{n}{2}-3}v_{\frac{n}{2}+1}$ are not. 

\item The edge $v_{\frac{n}{2}-3}v_{\frac{n}{2}+1}$ has two distinct paths of length one to a $C_3$, while the edge $v_2v_{n-4}$ only has one. 
\end{itemize}
\vspace{-5pt}
Thus, the edges $v_2v_{n-4}$ and $v_{\frac{n}{2}-3}v_{\frac{n}{2}+1}$ are different. In this way, all the edges in $E(\text{II})$ are different. 

\item $E(\text{III})$:
\vspace{-6pt}
\begin{itemize}
\item Edges $v_{\frac{n}{2}-1}v_{\frac{n}{2}}$ and $v_{\frac{n}{2}-1}v_{\frac{n}{2}-2}$ are in neither a $C_4$ nor a $C_6$, while the edge $v_{n-1}v_{n-3}$ is in a $C_6$ (and not in a $C_4$), and the edge $v_{n-2}v_{n-1}$ is in a $C_4$ (and not in a $C_6$). In this way, the edges $v_{n-1}v_{n-3}$ and $v_{n-2}v_{n-1}$ are different. 

\item In order to differentiate between the edges $v_{\frac{n}{2}-1}v_{\frac{n}{2}}$ and $v_{\frac{n}{2}-1}v_{\frac{n}{2}-2}$, we compare the distance from each edge to the $C_3$ subgraph with vertices $v_{n-1}, v_{n-2}, v_{n-3}$. When comparing distances to the $C_3$ subgraph between any edge $v_k v_t$, where $2 \leq k \leq n-5$ and $t=k+1$, we do not consider paths that use the edge $v_nv_{\frac{n}{2}-1}$, as two edges can reach the $C_3$ subgraph in the same distance using this edge (due to the symmetry below the edge $v_2v_{n-4}$). Therefore, when we compare the distance from each edge to the $C_3$ subgraph, we use paths that only use the exterior edges below the edge $v_2v_{n-4}$ in order to differentiate between edges. 

Using paths that only use the exterior edges below the edge $v_2v_{n-4}$, we find that the edge $v_{\frac{n}{2}-1}v_{\frac{n}{2}}$ is distance four away from the $C_3$ subgraph, while the edge $v_{\frac{n}{2}-1}v_{\frac{n}{2}-2}$ is distance three. Therefore, the edges $v_{\frac{n}{2}-1}v_{\frac{n}{2}}$ and $v_{\frac{n}{2}-1}v_{\frac{n}{2}-2}$ are different. Thus, all the edges in $E(\text{III})$ are different. 
\end{itemize}
\vspace{-10pt}
\item $E(\text{IV})$: 
\vspace{-7pt}
\begin{itemize}
\item Edge $v_{\frac{n}{2}-2}v_{\frac{n}{2}}$ is part of a $C_4$, while the edge $v_{n-3}v_{n-2}$ is not. In this way, the edges in $E(\text{IV})$ are different. 
\end{itemize}
\vspace{-10pt}
\item $E(\text{V})$: 
\vspace{-7pt}
\begin{itemize}
\item Edge $v_nv_{\frac{n}{2}-1}$ has a path of length one to a $C_3$ subgraph, while the edge $v_{n-3}v_{n-4}$ does not. Thus, the edges in $E(\text{V})$ are different. 
\end{itemize}
\vspace{-10pt}
\item $E(\text{VI})$:
\vspace{-5pt}
\begin{itemize}
\item We differentiate between edges within this group by comparing distances from each edge to both $C_3$ subgraphs. Recall that edges below the edge $v_2v_{n-4}$ have symmetry about the edge $v_nv_{\frac{n}{2}-1}$. Thus, when comparing distances from edges to each $C_3$ subgraph, we use paths that do not use the edge $v_nv_{\frac{n}{2}-1}$ and only use the exterior edges below the edge $v_2v_{n-4}$ to ensure distinct path lengths. 

For each edge in set VI, when we list the shortest path distance to each $C_3$ subgraph, we find that each edge has a unique pair of distances. This is because the edge $v_nv_{\frac{n}{2}-1}$ splits $C_n$ into two unequal subgraphs so that the path along the exterior of one side of the edge $v_nv_{\frac{n}{2}-1}$ to the top $C_3$ subgraph is different to the path along the exterior of the other side of the edge $v_nv_{\frac{n}{2}-1}$. This ensures that all edges in $E(\text{VI})$ are different. 
\end{itemize}
\vspace{-10pt}
\item $E(\text{VII})$:
\vspace{-5pt}
\begin{itemize}
\item We differentiate between edges within this group by comparing distances from each edge to both $C_3$ subgraphs. Recall that edges below the edge $v_2v_{n-4}$ have symmetry about the edge $v_nv_{\frac{n}{2}-1}$. Thus, when comparing distances from edges to each $C_3$ subgraph, we use paths that do not use the edge $v_nv_{\frac{n}{2}-1}$ and only use the exterior edges below the edge $v_2v_{n-4}$  to ensure distinct path lengths. For each edge in Group VII, when we list the shortest path distance to each $C_3$ subgraph, we find that each edge has a distinct pair of distances. This is because the edge $v_nv_{\frac{n}{2}-1}$ splits $C_n$ into two unequal subgraphs so that the path along the exterior of one side of the edge $v_n v_{\frac{n}{2}-1}$ to the top $C_3$ subgraph is different to the path along the exterior of the other side of the edge $v_nv_{\frac{n}{2}-1}$. This ensures that all edges in $E(\text{VII})$ are different.




\end{itemize}

\end{itemize}
\hspace{10pt} In this way, the edges within each group are different from each other. Thus, every edge is different compared to all other edges in the graph. Consequently, the graph has a trivial automorphism group, and therefore, has no symmetry. In this way, we have shown that this infinite family of cubic Hamiltonian graphs on even $n \geq 12$ vertices has no symmetry. 

\section{Asymmetric Quartic Hamiltonian Graphs}
\hspace{10pt} We now present constructions of asymmetric quartic Hamiltoninan graphs on even $n \geq 12$ vertices. These constructions depend on congruency of $n$ modulo 4. 

\subsection*{Constructing the graph on $n \equiv 0 \; \text{mod} \; 4$ vertices}
\begin{itemize}
\item[(1)] Construct the cubic Hamiltonian graph described in Figure 1 (a) starting at $C_n$ and label the vertices from 1 to $n$ consecutively, traversing clockwise.
\vspace{-5pt}
\item[(2)] Add the following edges:
\begin{figure}[h!]
\centering
\begin{tabular}{lll}
$v_{\frac{n}{4}}v_{\frac{3n}{4}}$ & and & $v_{\frac{n}{4}+i}v_{(\frac{n}{4}-i) \; \text{mod} \; n}$ 
where $1 \leq i \leq \frac{n}{2}-1$
\end{tabular} 
\end{figure}

\end{itemize}
\subsection*{Constructing the graph on $n \equiv 2 \; \text{mod} \; 4$ vertices}
\begin{itemize}
\item[(1)] Construct the cubic Hamiltonian graph described in Figure 4 (b) starting at $C_n$ and label the vertices from 1 to $n$ consecutively, traversing clockwise.
\vspace{-5pt}
\item[(2)] Add the following edges:
\begin{figure}[h!]
\centering
\begin{tabular}{lll}
$v_{1} v_{\frac{n}{2}+1}$ & and & $v_{1+i}v_{(1-i) \; \text{mod} \; n}$ 
where $1 \leq i \leq \frac{n}{2}-1$
\end{tabular} 
\end{figure}
\end{itemize}
Examples of these graphs on $n=12$ and $n=18$ vertices are depicted in Figure 5. 
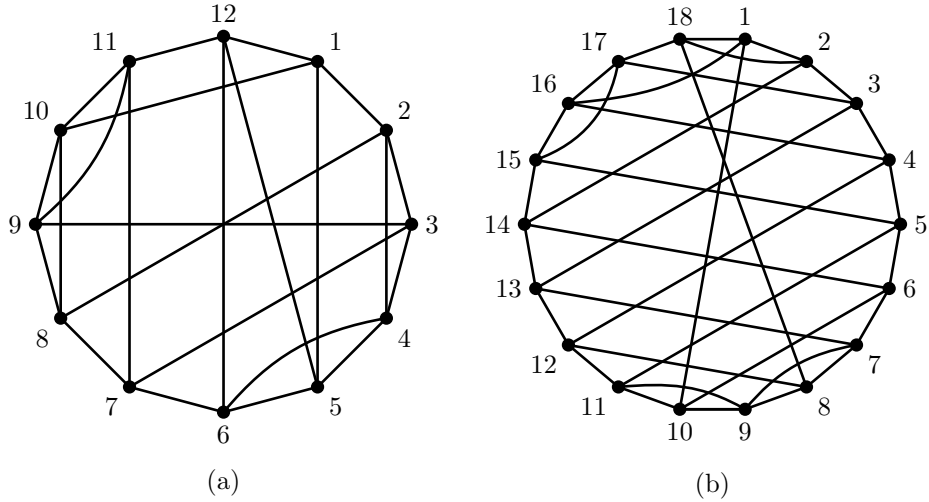
\begin{figure}[h!]
\begin{center}
\begin{tikzpicture}[node distance = 0.1cm, line width = 1pt]
\coordinate (1) at ({2.5*cos(30)},{2.5*sin(30)});
\coordinate (2) at ({2.5*cos(60)},{2.5*sin(60)});
\coordinate (3) at ({2.5*cos(90)},{2.5*sin(90)});
\coordinate (4) at ({2.5*cos(120)},{2.5*sin(120)});
\coordinate (5) at ({2.5*cos(150)},{2.5*sin(150)});
\coordinate (6) at ({2.5*cos(180)},{2.5*sin(180)});
\coordinate (7) at ({2.5*cos(210)},{2.5*sin(210)});
\coordinate (8) at ({2.5*cos(240)},{2.5*sin(240)});
\coordinate (9) at ({2.5*cos(270)},{2.5*sin(270)});
\coordinate (10) at ({2.5*cos(300)},{2.5*sin(300)});
\coordinate (11) at ({2.5*cos(330)},{2.5*sin(330)});
\coordinate (12) at ({2.5*cos(0)},{2.5*sin(0)});
\node (13) [above right = 0.25mm of 1] {2};
\node (14) [above right = 0.25mm of 2] {1};
\node (15) [above = 0.5mm of 3] {12};
\node (16) [above left = 0.25mm of 4] {11};
\node (17) [above left = 0.25mm of 5] {10};
\node (18) [left = 0.5mm of 6] {9};
\node (19) [below left = 0.25mm of 7] {8};
\node (20) [below left = 0.25mm of 8] {7};
\node (21) [below = 0.5mm of 9] {6};
\node (22) [below right = 0.25mm of 10] {5};
\node (23) [below right = 0.25mm of 11] {4};
\node (24) [right = 0.5mm of 12] {3};

\coordinate (25) at ({6.5 + 2.5*cos(20)},{2.5*sin(20)});
\coordinate (26) at ({6.5 + 2.5*cos(40)},{2.5*sin(40)});
\coordinate (27) at ({6.5 + 2.5*cos(60)},{2.5*sin(60)});
\coordinate (28) at ({6.5 + 2.5*cos(80)},{2.5*sin(80)});
\coordinate (29) at ({6.5 + 2.5*cos(100)},{2.5*sin(100)});
\coordinate (30) at ({6.5 + 2.5*cos(120)},{2.5*sin(120)});
\coordinate (31) at ({6.5 + 2.5*cos(140)},{2.5*sin(140)});
\coordinate (32) at ({6.5 + 2.5*cos(160)},{2.5*sin(160)});
\coordinate (33) at ({6.5 + 2.5*cos(180)},{2.5*sin(180)});
\coordinate (34) at ({6.5 + 2.5*cos(200)},{2.5*sin(200)});
\coordinate (35) at ({6.5 + 2.5*cos(220)},{2.5*sin(220)});
\coordinate (36) at ({6.5 + 2.5*cos(240)},{2.5*sin(240)});
\coordinate (37) at ({6.5 + 2.5*cos(260)},{2.5*sin(260)});
\coordinate (38) at ({6.5 + 2.5*cos(280)},{2.5*sin(280)});
\coordinate (39) at ({6.5 + 2.5*cos(300)},{2.5*sin(300)});
\coordinate (40) at ({6.5 + 2.5*cos(320)},{2.5*sin(320)});
\coordinate (41) at ({6.5 + 2.5*cos(340)},{2.5*sin(340)});
\coordinate (42) at ({6.5 + 2.5*cos(0)},{2.5*sin(0)});
\node (43) [above = 0.5mm of 29] {18};
\node (44) [above = 0.5mm of 28] {1};
\node (45) [above right = 0.25mm of 27] {2};
\node (46) [above right = 0.25mm of 26] {3};
\node (47) [right = 0.5mm of 25] {4};
\node (48) [right = 0.5mm of 42] {5};
\node (49) [right = 0.5mm of 41] {6};
\node (50) [below right = 0.25mm of 40] {7};
\node (51) [below right = 0.25mm of 39] {8};
\node (52) [below = 0.5mm of 38] {9};
\node (53) [below = 0.5mm of 37] {10};
\node (54) [below left = 0.25mm of 36] {11};
\node (55) [below left = 0.25mm of 35] {12};
\node (56) [left = 0.5mm of 34] {13};
\node (57) [left = 0.5mm of 33] {14};
\node (58) [left = 0.5mm of 32] {15};
\node (59) [above left = 0.25mm of 31] {16};
\node (60) [above left = 0.25mm of 30] {17};
\node (61) [below = 0.75mm of 21] {(a)};

\draw \foreach \x [remember=\x as \lastx (initially 1)] in {2,3,4,5,6,7,8,9,10,11,12,1}{(\lastx) -- (\x)};
\draw \foreach \x [remember=\x as \lastx (initially 25)] in {26,27,28,29,30,31,32,33,34,35,36,37,38,39,40,41,42,25}{(\lastx) -- (\x)};
\draw (6) to [out=40, in=-100] (4);
\draw (9) to [out=50, in=190] (11);
\draw (3)--(10);
\draw (1)--(7);
\draw (12)--(8);
\draw (2)--(5);
\draw (12)--(6);
\draw (1)--(11);
\draw (2)--(10);
\draw (3)--(9);
\draw (4)--(8);
\draw (5)--(7);

\draw (29)--(39);
\draw (27)--(33);
\draw (26)--(34);
\draw (25)--(35);
\draw (42)--(36);
\draw (41)--(37);
\draw (28)--(37);
\draw (35)--(39);
\draw (34)--(40);
\draw (33)--(41);
\draw (32)--(42);
\draw (31)--(25);
\draw (30)--(26);

\draw (31) to [out=5, in=-140] (28);
\draw (32) to [out=20, in=-100] (30);
\draw (38) to [out=50, in=190] (40);
\draw (29) to [out=-20, in=185] (27);
\draw (38) to [out=150, in=5] (36);
\draw (37) to node [below=7mm]{(b)} (38);

\foreach \point in {1,2,3,4,5,6,7,8,9,10,11,12,25,26,27,28,29,30,31,32,33,34,35,36,37,38,39,40,41,42} \fill (\point) circle (2.5pt);

\end{tikzpicture} 
\end{center}
\caption{Asymmetric quartic Hamiltonian graphs}
\end{figure} \\

\noindent We first show asymmetry for the cases of small $n$; that is, $n=12,14,16, \text{and} \; 18$.  

\subsection{Quartic Hamiltonian Graph on 12 Vertices}
We divide $V(G)$ into the following subsets:

\begin{itemize}
\item $V(\text{I})=\{v_2, v_3, v_4\}$\\
\vspace{-5pt} These vertices form a $C_3$ that is one away from two $C_3$s and shares a vertex with another $C_3$.

\item $V(\text{II})=\{v_1, v_{12}, v_5\}$\\
\vspace{-5pt} These vertices form a $C_3$ that is one away from two $C_3$s and shares an edge with another $C_3$.

\item $V(\text{III})=\{v_{12}, v_5, v_6\}$\\
\vspace{-5pt} These vertices form a $C_3$ that shares a common edge with two $C_3$s.

\item $V(\text{IV})=\{v_4, v_5, v_6\}$\\
\vspace{-5pt} These vertices form a $C_3$ that shares vertices with three distinct $C_3$s.

\item $V(\text{V})=\{v_8, v_9, v_{10}\}$\\
\vspace{-5pt} These vertices form a $C_3$ that is one away from two $C_3$s.

\item $V(\text{VI})=\{v_7, v_{11} \}$\\
\vspace{-5pt} These vertices are not in any $C_3$s.
\end{itemize}
\vspace{3pt}
\hspace{10pt} In this way, vertices in any given subset are different from vertices that are not in that same subset. In order to show that every vertex in $V(\text{G})$ is unique, we must also show that each vertex in every subset is different from all other vertices within the same subset.

\begin{itemize}
\item $V(\text{I})$:
\vspace{-7pt}
\begin{itemize}
\item $v_2$ is distance one away from $V(\text{II})$.
\vspace{-5pt}
\item $v_3$ is not distance one away from $V(\text{II})$
\vspace{-5pt}
\item $v_4$ is distance one away from $V(\text{II})$ and in $V(\text{IV})$.
\vspace{-5pt}
\end{itemize}

\item $V(\text{II})$:
\vspace{-7pt}
\begin{itemize}
\item $v_{12}$ is in both $V(\text{III})$ and $V(\text{II})$.
\vspace{-5pt}
\item $v_5$ is in $V(\text{III})$, $V(\text{IV})$ and $V(\text{II})$.
\vspace{-5pt}
\item $v_1$ is in only $V(\text{II})$.
\vspace{-5pt}
\end{itemize}

\item $V(\text{III})$:
\vspace{-7pt}
\begin{itemize}
\item $v_{12}$ is in $V(\text{II})$ and $V(\text{III})$.
\vspace{-5pt}
\item $v_5$ is in $V(\text{II})$, $V(\text{IV})$ and $V(\text{III})$.
\vspace{-5pt}
\item $v_6$ is in $V(\text{IV})$ and  $V(\text{III})$.
\vspace{-5pt}
\end{itemize}

\item $V(\text{IV})$:
\vspace{-7pt}
\begin{itemize}
\item $v_4$ is distance one away from $V(\text{II})$ and in $V(\text{IV})$.
\vspace{-5pt}
\item $v_5$ is in $V(\text{II})$, $V(\text{IV})$ and $V(\text{III})$.
\vspace{-5pt}
\item $v_6$ is in $V(\text{IV})$ and $V(\text{III})$.
\vspace{-5pt}
\end{itemize}

\item $V(\text{V})$:
\vspace{-7pt}
\begin{itemize}
\item $v_{8}$ is distance two away from $v_3$, which we now know is unique.
\vspace{-5pt}
\item $v_{9}$ is distance one away from $v_3$, which is unique.
\vspace{-5pt}
\item $v_{10}$ is distance one away from $V(\text{II})$.
\vspace{-5pt}
\end{itemize}

\item $V(\text{VI})$:
\vspace{-7pt}
\begin{itemize}
\item $v_{7}$ is distance one away from $V(\text{IV})$
\vspace{-5pt}
\item $v_{11}$ is distance two away from $V(\text{IV})$
\vspace{-5pt}
\end{itemize}

\end{itemize}

\hspace{10pt} Since we have shown that each subset of $V(\text{G})$ is different from all other subsets, and since we have described how vertices within each subset are unique, all vertices in $V(\text{G})$ are unique and the graph is asymmetric. 
\subsection{Quartic Hamiltonian Graph on 16 Vertices}
We divide $V(G)$ into the following subsets:

\begin{itemize}
\item $V(\text{I})$ $=\{v_3, v_4, v_5\}$\\
\vspace{-5pt} These vertices form a $C_3$ that is one away from two $C_3$s.

\item $V(\text{II})$ $=\{v_{11}, v_{12}, v_{13}\}$\\
\vspace{-5pt} These vertices form a $C_3$ that is distance one away from two $C_3$s and is in another $C_3$.

\item $V(\text{III})$ $=\{v_{7}, v_8, v_{16}\}$\\
\vspace{-5pt} These vertices form a $C_3$ that shares a common edge with two $C_3$s.

\item $V(\text{IV})$ $=\{v_1, v_7, v_{16}\}$\\
\vspace{-5pt} These vertices form a $C_3$ that is two away from two distinct $C_3$s. 

\item $V(\text{V})$ $=\{v_{13}, v_{14}, v_{15}\}$\\
\vspace{-5pt} These vertices form a $C_3$ that shares a vertex, but does not share an edge with another $C_3$.

\item $V(\text{VI})$ $=\{v_6, v_{7} , v_8 \}$\\
\vspace{-1pt} These vertices form a $C_3$ that shares a vertex with two distinct $C_3$s, shares an edge with one $C_3$, and is one away from a $C_3$ whose vertices are all in only one $C_3$.
\vspace{-5pt}
\item $V(\text{VII})$ $=\{v_2, v_9, v_{10} \}$\\
\vspace{-5pt} These vertices are not in any $C_3$s.

\end{itemize}

\hspace{10pt} In this way, vertices in any given subset are different from vertices that are not in that same subset. In order to show that every vertex in $V(\text{G})$ is unique, we must also show that each vertex in every subset is different from all other vertices within the same subset.
\begin{itemize}
\item $V(\text{III})$:
\vspace{-7pt}
\begin{itemize}
\item $v_7$ is in $V(\text{VI})$, $V(\text{IV})$, and $V(\text{III})$.
\vspace{-5pt}
\item $v_8$ is in both $V(\text{VI})$ and $V(\text{III})$.
\vspace{-5pt}
\item $v_{16}$ is not in $V(\text{VI})$.
\vspace{-5pt}
\end{itemize}

\item $V(\text{IV})$:
\vspace{-7pt}
\begin{itemize}
\item $v_{16}$ is in $V(\text{III})$ and $V(\text{IV})$.
\vspace{-5pt}
\item $v_7$ is in $V(\text{IV})$, $V(\text{VI})$, and $V(\text{III})$.
\vspace{-5pt}
\item $v_1$  is not in $V(\text{III})$ nor $V(\text{IV})$.
\vspace{-5pt}
\end{itemize}

\item $V(\text{V})$:
\vspace{-7pt}
\begin{itemize}
\item $v_13$ is in both $V(\text{II})$ and $V(\text{V})$.
\vspace{-5pt}
\item $v_14$ is distance two away from $V(\text{III})$.
\vspace{-5pt}
\item $v_15$ is distance one away from $V(\text{III})$.
\vspace{-5pt}
\end{itemize}

\item $V(\text{VI})$:
\vspace{-7pt}
\begin{itemize}
\item $v_8$ is in both $V(\text{III})$ and $V(\text{VI})$.
\vspace{-5pt}
\item $v_7$ is in $V(\text{IV})$, $V(\text{III})$, and $V(\text{VI})$.
\vspace{-5pt}
\item $v_6$ is not in $V(\text{IV})$ nor $V(\text{III})$.
\vspace{-5pt}
\end{itemize}

\item $V(\text{VII})$:
\vspace{-7pt}
\begin{itemize}
\item $v_2$ is distance one away from $V(\text{VI})$ and one away from $V(\text{IV})$.
\vspace{-5pt}
\item $v_9$ is distance one away from $V(\text{VI})$ and one away from  $V(\text{V})$.
\vspace{-5pt}
\item $v_{10}$ is distance two away from $V(\text{VI})$.
\vspace{-5pt}
\end{itemize}

\item $V(\text{I})$:
\vspace{-7pt}
\begin{itemize}
\item $v_3$ is distance one away from $v_2$, which we now know is unique.
\vspace{-5pt}
\item $v_4$ is distance two away from $v_2$, which we know is unique.
\vspace{-5pt}
\item $v_5$ is distance one away from $V(\text{VI})$.
\vspace{-5pt}
\end{itemize}

\item $V(\text{II})$:
\vspace{-7pt}
\begin{itemize}
\item $v_{11}$ is distance two away from $v_2$, which we now know is unique.
\vspace{-5pt}
\item $v_{12}$ is distance one away from $v_2$, which we know is unique.
\vspace{-5pt}
\item $v_{13}$ is in both $V(\text{V})$ and $V(\text{II})$.
\vspace{-5pt}
\end{itemize}

\end{itemize}

\hspace{10pt} Since we have shown that each subset of $V(\text{G})$ is different from all other subsets, and since we have described how vertices within each subset are unique, all vertices in $V(\text{G})$ are unique and the graph is asymmetric.


\subsection{Quartic Hamiltonian Graph on 14 Vertices}
We distinguish between vertices as follows:
\vspace{-5pt}
\begin{itemize}
\item The vertices $v_1, \; v_2, \; v_{5}, \; v_{6}, \; v_{7}, \; v_{8}, \; v_{9}, \; v_{11}, \; v_{12}, \; v_{13}, \; v_{14}$ are the only vertices contained in a $C_3$ subgraph:

\begin{itemize}
\item \vspace{-5pt} $v_{7}$ is unique because it is the only vertex contained in two distinct $C_3$s.

\item \vspace{-5pt}  $v_2$ and $v_{11}$ are the only vertices contained in a $C_3$ that is not distance one away from another $C_3$.
\begin{itemize}
\item \vspace{-3pt} $v_2$ is distance three away from the vertex $v_{7}$, while the vertex $v_{11}$ is not. 

\end{itemize}
\end{itemize}
\end{itemize}

\begin{itemize}
\item \vspace{-5pt} The vertices $v_1, \; v_{6}, \; v_{8}, \; v_{12}, \; v_{13}, \; v_{5}, \; v_{9} \; \text{and} \; v_{14}$ are all contained in a $C_3$ and are one away from a $C_3$:

\begin{itemize}
\item \vspace{-5pt} $v_1$ and $v_{14}$ are adjacent to $v_2$.

\begin{itemize}
\item \vspace{-3pt} $v_1$ is only contained in one $C_4$ whereas $v_{14}$ is contained in two $C_4$s. 
\end{itemize}

\item \vspace{-3pt} $v_{6}$, $v_{8}$, $v_{5}$, and \;$ v_{9}$ are adjacent to the vertex $v_{7}$.

\begin{itemize}
\item \vspace{-3pt} $v_{6}$ is adjacent to $v_{14}$, while $v_{8}$, $v_{5}$, and \;$ v_{9}$ are not.

\item \vspace{-3pt} $v_{8}$ is adjacent to $v_1$, while $v_{6}$, $v_{5}$, and \;$ v_{9}$ are not.

\item \vspace{-3pt} $v_{5}$ is distance two away from $v_{14}$, while $v_{9}$ is not.
\end{itemize}

\item \vspace{-3pt} $v_{13}$ and $v_{12}$ are not adjacent to $v_{7}$ nor $v_2$.

\begin{itemize}
\item \vspace{-3pt} $v_{13}$ is adjacent to $v_{14}$, while the vertex $v_{12}$ is not.
\end{itemize}
\end{itemize}

\item \vspace{-5pt} The vertices $v_3$, $v_4$ and $v_{10}$ are not in a $C_3$:
\begin{itemize}
\item \vspace{-5pt} $v_{3}$ is adjacent to $v_{13}$.
\item \vspace{-5pt} $v_{4}$ is adjacent to $v_{8}$.
\item \vspace{-5pt} $v_{10}$ is adjacent to $v_{11}$.
\end{itemize}

\end{itemize}

\noindent \vspace{-5pt} In this way all vertices are unique and the graph is asymmetric.


\subsection{Quartic Hamiltonian Graph on 18 Vertices}

We construct a quartic Hamiltonian graph (as described in section 1.1). 
\noindent We distinguish between vertices as follows:

\begin{itemize}

\vspace{-5pt}
\item Vertices $v_1, \; v_2, \; v_{7}, \; v_{8}, \; v_{9}, \; v_{10}, \; v_{11}, \; v_{15}, \; v_{16}, \; v_{17}$, and $\; v_{18}$ are the only vertices contained in any $C_3$ subgraph:

\begin{itemize}
\item \vspace{-5pt} The vertex $v_{9}$ is unique because it is the only vertex contained in two distinct $C_3$s.

\item \vspace{-5pt} The vertices $v_2$ and $v_{15}$ are the only vertices contained in a $C_3$ that is not one away from another $C_3$.

\begin{itemize}
\item \vspace{-5pt} $v_2$ is distance three away from the vertex $v_{9}$, while the vertex $v_{15}$ is not. 
\end{itemize}
\end{itemize}

\item \vspace{-5pt} The vertices $v_1, \; v_{8}, \; v_{10}, \; v_{16}, \; v_{17}, \; v_{7}, \; v_{11} \; \text{and} \; v_{18}$ are all contained in a $C_3$ and are one away from a $C_3$:

\begin{itemize}
\item \vspace{-5pt}  $v_1$ and $v_{18}$ are adjacent to $v_2$.

\begin{itemize}
\item \vspace{-3pt} $v_1$ is only contained in one $C_4$ whereas $v_{18}$ is contained in two $C_4$s. 
\end{itemize}

\item \vspace{-5pt} $v_{8}$, $v_{10}$, $v_{7}$, and \;$ v_{11}$ are adjacent to the vertex $v_{9}$.

\begin{itemize}
\item \vspace{-3pt} $v_{8}$ is adjacent to $v_{18}$, while $v_{10}$, $v_{7}$, and \;$ v_{11}$ are not.

\item \vspace{-3pt} $v_{10}$ is adjacent to $v_1$, while $v_{8}$, $v_{7}$, and \;$ v_{11}$ are not.

\item \vspace{-3pt} $v_{7}$ is distance two away from $v_{18}$, while $v_{11}$ is not.
\end{itemize}

\item \vspace{-5pt} $v_{17}$ and $v_{16}$ are not adjacent to $v_{9}$ nor $v_2$.

\begin{itemize}
\item \vspace{-5pt} $v_{17}$ is adjacent to $v_{18}$, while the vertex $v_{16}$ is not.
\end{itemize}
\end{itemize}

\item \vspace{-5pt} The vertices $v_3$, $v_4$, $v_5$, $v_6$, $v_{12}$, $v_{13}$ and $v_{14}$ are not in any $C_3$s and we distinguish between them as follows:

\begin{itemize}
\item \vspace{-5pt} $v_3$: is adjacent to $v_{17}$, $v_{2}$, $v_{4}$ and $v_{13}$. We know $v_{17}$ and $v_{2}$ are unique. $v_{4}$ is adjacent to $v_{16}$ and $v_{13}$ is not. Thus, $v_{16}$ and $v_{13}$ are unique, and we deduce that $v_3$ is unique.

\item \vspace{-5pt} $v_4$: is adjacent to $v_{16}$, $v_{3}$, $v_{5}$ and $v_{12}$. We know $v_{16}$ and $v_{3}$ are unique. $v_{5}$ is adjacent to $v_{15}$ and $v_{12}$ is not. Thus, $v_{15}$ and $v_{12}$ are unique, and we deduce that $v_4$ is unique.

\item \vspace{-5pt} $v_5$: is adjacent to $v_{15}$, $v_{4}$, $v_{6}$ and $v_{11}$. We know $v_{15}$, $v_{4}$, and $v_{11}$ are unique. Thus, we deduce that $v_{6}$ is unique. Therefore, $v_5$ is unique.

\item \vspace{-5pt} $v_6$: is adjacent to $v_{14}$, $v_{5}$, $v_{10}$ and $v_{7}$, which are all unique. Thus, $v_6$ is unique.

\end{itemize}
In this way the vertices $v_3$, $v_4$, $v_5$, $v_6$, $v_{12}$, $v_{13}$ and $v_{14}$ are all unique. Therefore, all vertices are unique and the graph is asymmetric.
\end{itemize}

\subsection{Quartic Hamiltonian for Large $n$}

We now show asymmetry for all $n \geq 20$. 

\subsubsection{When $n \equiv 0  \; \text{mod} \; 4$}
\hspace{10pt} In order to show that a quartic Hamiltonian graph on $n \geq 20 $, $n \equiv 0  \; \text{mod} \; 4$ vertices has a trivial automorphism group (no symmetry), we must show that every vertex in the graph is unique. We do this by comparing every vertex in the graph to every other vertex. We begin by grouping vertices based on differentiating properties, so that vertices in each group are different from vertices in all other groups. \\

\textbf{Vertices that are Always Contained in a 3-Cycle}\\

We find that the graph will always have six distinct $C_3$s, no matter how large $n$ is. These six $C_3$s will always have the same distinct properties, making each $C_3$ different from every other $C_3$. The vertex sets for these $C_3$s and their distinguishing characteristics are listed below. It is important to note that the edge  $v_{\frac{n}{4}} v_{\frac{3n}{4}}$ is unique because it is the only edge that connects two $C_3$s whose vertices are all in only one $C_3$, and all have a distance greater than one to another $C_3$.

\begin{itemize}
\item \vspace{-3pt} $V(\text{I}) = \{v_{\frac{n}{2}}, v_{\frac{n}{2}-1}, v_{\frac{n}{2}-2}\}$ \\
\vspace{-3pt} These vertices form a $C_3$ that shares a common vertex with another triangle and shares a common edge with another triangle. It is also at least distance two from any other triangle.
\vspace{-3pt}
\item \vspace{-3pt} $V(\text{II}) = \{v_n, v_1,v_{\frac{n}{2}-1}\}$\\
\vspace{-5pt} These vertices form a $C_3$ that is distance one away from a $C_3$ whose vertices are all in only one $C_3$.
\vspace{-3pt}
\item $V(\text{III}) = \{v_n, v_{\frac{n}{2}},v_{\frac{n}{2}-1}\}$ \\
\vspace{-5pt} These vertices form a $C_3$ that shares a common edge with two other triangles.
\vspace{-3pt}
\item $V(\text{IV}) = \{v_{n-1}, v_{n-2}, v_{n-3}\}$ \\
\vspace{-5pt} These vertices form a $C_3$ that are distance one from vertices in two triangles.
\vspace{-3pt}
\item $V(\text{V}) = \{v_{\frac{n}{4}}, v_{\frac{n}{4}-1}, v_{\frac{n}{4}+1}\}$ \\
\vspace{-3pt} These vertices form a $C_3$ that, compared to $V(\text{VI})$, $V(\text{V})$, has a greater number of distinct shortest paths to vertices in $V(\text{I})$, using the special edge $v_{\frac{n}{4}} v_{\frac{3n}{4}}$ first.
\vspace{-5pt}
\item $V(\text{VI}) = \{v_{\frac{3n}{4}}, v_{\frac{3n}{4}-1}, v_{\frac{3n}{4} +1}\}$ \\
\vspace{-3pt} These vertices form a $C_3$ that, compared to $V(\text{V})$,$V(\text{VI})$, has a smaller number of distinct shortest paths to vertices in $V(\text{I})$, using the special edge $v_{\frac{n}{4}} v_{\frac{3n}{4}}$ first.
\end{itemize}
\hspace{10pt} In this way, any vertex in any $C_3$ is different from all other vertices in any other $C_3$. Now we must show that each vertex in each $C_3$ is different from the vertices within the same $C_3$.

\begin{itemize}
\item $V(\text{I})$:
\vspace{-7pt}
\begin{itemize}
\item $v_{\frac{n}{2}}$ is contained in two $C_3s$.
\vspace{-5pt}
\item $v_{\frac{n}{2}-1}$ is contained in three $C_3s$.
\vspace{-5pt}
\item $v_{\frac{n}{2}-2}$ is contained in one $C_3$. 
\end{itemize}
\vspace{-5pt}
\item $V(\text{II})$: 
\vspace{-7pt}
\begin{itemize}
\item $v_n$ is distance two away from $v_2$
\vspace{-5pt}
\item $v_1$ is distance one away from $v_2$
\vspace{-5pt}
\item $v_{\frac{n}{2}-1}$ is contained in 3 $C_3s$
\end{itemize}

\item $V(\text{III})$: 
\vspace{-7pt}
\begin{itemize}
\item $v_n$ is distance two away from $v_2$
\vspace{-5pt}
\item $v_{\frac{n}{2}}$ is distance three away from $v_2$ 
\vspace{-5pt}
\item $v_{\frac{n}{2}-1}$ is contained in 3 $C_3s$.
\end{itemize}

\item $V(\text{IV})$: Each vertex in this vertex set is differentiated by a unique ordered pair $(x,y)$, where $x$ is the distance to $v_{\frac{n}{2}-1}$ and $y$ is the distance to $v_2$
\vspace{-7pt}
\begin{itemize}
\item $v_{n-1}$: $(2,3)$
\vspace{-5pt}
\item $v_{n-2}$: $(2,2)$
\vspace{-5pt}
\item $v_{n-3}$: $(3,2)$
\end{itemize}
\vspace{-5pt}

\item $V(\text{V})$:
\vspace{-7pt}
\begin{itemize}
\item $v_{\frac{n}{4}}$ is adjacent to the special edge $v_{\frac{n}{4}}v_{\frac{3n}{4}}$ 
\vspace{-3pt}
\item $v_{\frac{n}{4}-1}$ is always has a greater distance to $v_2$ than $v_{\frac{n}{4}+1}$ 
\vspace{-3pt}
\item $v_{\frac{n}{4}+1}$ is always closer to $v_2$ than $v_{\frac{n}{4}-1}$. 
\end{itemize}
\vspace{-5pt}
\item $V(\text{VI})$:  
\vspace{-7pt}
\begin{itemize}
\item $v_{\frac{3n}{4}}$ is adjacent to the special edge $v_{\frac{n}{4}}v_{\frac{3n}{4}}$
\vspace{-3pt}
\item $v_{\frac{3n}{4}-1}$ is always has a greater distance to $v_2$ than $v_{\frac{3n}{4}+1}$ 
\vspace{-3pt}
\item $v_{\frac{3n}{4} +1}$ is always closer to $v_2$ than $v_{\frac{3n}{4}-1}$
\end{itemize}
\end{itemize}


\textbf{Vertices That Are Not Contained in a 3-cycle}
\newline We can divide the vertices that are not contained in a 3-cycle into two sets. 

\begin{itemize}
\item $V(\text{VII})$ = $\{ v_2, v_3, v_4, v_{\frac{n}{2}-4}, v_{\frac{n}{2}-3}, v_{\frac{n}{2}+1}, v_{\frac{n}{2}+2}, v_{\frac{n}{2}+3}\}$
\vspace{-3pt}
\begin{itemize}
\item \vspace{-3pt} $v_2$ is distance one away from $V(\text{I})$ and one away from $V(\text{II})$. 
\item \vspace{-3pt} $v_3$ is distance two away from $V(\text{II})$, distance three away from $V(\text{IV})$ and distance one away from $v_2$.
\item \vspace{-3pt} $v_4$ is distance four away from $V(\text{III})$ and distance three away from $V(\text{IV})$.
\item \vspace{-3pt} $v_{\frac{n}{2}-4}$ is distance two away from $V(\text{I})$, and distance two away from $V(\text{IV})$. 
\item \vspace{-3pt} $v_{\frac{n}{2}-3}$ is distance one away from $V(\text{I})$.
\item \vspace{-3pt} $v_{\frac{n}{2}+1}$ is distance one away from $V(\text{I})$ and distance one away from $V(\text{IV})$. 
\item \vspace{-3pt} $v_{\frac{n}{2}+2}$ is distance one away from $V(\text{IV})$, and distance two away from $V(\text{II})$.
\item \vspace{-3pt} $v_{\frac{n}{2}+3}$ is distance one away from $V(\text{IV})$, and distance three away from $V(\text{II})$.

\noindent In this way the vertices in Set VII are all unique. 
\end{itemize}

\item $V(\text{VIII})$: Vertices that are added as the graph gets larger.
We can divide the remaining vertices, or the vertices that are added as the graph becomes larger, into the following vertex sets: 

\begin{itemize}

\item $ A= \{ \; v_x \in G\; : \; 4<x<\frac{n}{4}\; \}$

\item $ B= \{ \; v_x \in G\; : \; \frac{n}{4}<x< \frac{n}{2}-4 \; \}$

\item $ C = \{ \; v_x \in G\; : \; \frac{n}{2}+3<x<\frac{3n}{4} \; \}$

\item $ D = \{ \; v_x \in G\; : \; \frac{3n}{4}<x<n-3 \; \}$
\end{itemize}

In order to show that vertices in $V(A)$, $V(B)$, $V(C)$, and $V(D)$ are unique, we must compare vertices in each vertex set to vertices in every other vertex set. We must also compare every vertex in each set to all other vertices within that same set. 

\begin{itemize}
\item[$\circ$] Comparing Sets A and B to Sets C and D:

When comparing a vertex in set $A \cup B$ to a vertex in set $C \cup D$, a vertex in $A \cup B$ has either the same distance to the vertex $v_{\frac{n}{4}}$ as a vertex in $C \cup D$ has to $v_{\frac{3n}{4}}$, or a vertex in $A \cup B$ has a different distance to the vertex $v_{\frac{n}{4}}$ as a vertex in $C \cup D$ has to $v_{\frac{3n}{4}}$. As a result, we have two cases. 

\begin{itemize}
\item Case I:
A vertex in $A \cup B$ is the same distance to the vertex $v_{\frac{n}{4}}$, as a vertex in $C \cup D$ is to $v_{\frac{3n}{4}}$. 

Recall that the vertex $v_{\frac{n}{4}}$ is different from the vertex $v_{\frac{3n}{4}}$ because the vertex $v_{\frac{n}{4}}$ has a greater number of shortest paths to $V(\text{I})$ that begin with the edge $v_{\frac{n}{4}}v_{\frac{3n}{4}}$, compared to the vertex $v_{\frac{3n}{4}}$. (Note that the edge $v_{\frac{n}{4}}v_{\frac{3n}{4}}$ is unique because it connects two $C_3$s whose vertices are all in only one $C_3$, and all have a distance greater than one to another $C_3$). When comparing two vertices that are equidistant from $v_{\frac{n}{4}}$ and $v_{\frac{3n}{4}}$ respectively, we say that each vertex is $i$ away from $v_{\frac{n}{4}}$ and $v_{\frac{3n}{4}}$, respectively. Let $i=|\frac{n}{4}-x|$. Any $v_x$ distance $i$ away from $v_{\frac{n}{4}}$ or $v_{\frac{3n}{4}}$ will have a path to $V(\text{I})$ that passes through the edge $v_{\frac{n}{4}}v_{\frac{3n}{4}}$ in the $i+1$ position of the path. If we ensure that the edge $v_{\frac{n}{4}}, v_{\frac{3n}{4}}$ is used at the $i+1$ position of the path to $V(\text{I})$, we know that a vertex in $A \cup B$ will have a greater number of shortest paths that use the edge $v_{\frac{n}{4}}v_{\frac{3n}{4}}$ at the  $i+1$ position of the path to $V(\text{I})$, than the vertices in $C \cup D$. In this way, vertices in $A \cup B$ that are $i$ away from $v_{\frac{n}{4}}$ are unique from vertices in $C \cup D$ that are $i$ away from  $v_{\frac{3n}{4}}$. 

\item Case II: A vertex in $A \cup B$ is a different distance from the vertex $v_{\frac{n}{4}}$, than a vertex in $C \cup D$ is to $v_{\frac{3n}{4}}$.

Let a vertex in $C \cup B$ be distance $l$ from $v_{\frac{3n}{4}}$, where $l=|\frac{3n}{4}-x|$, and let a vertex in $A \cup B$ be distance $k$ from $v_{\left\lceil \frac{n}{4}\right\rceil }$, where $k=|\frac{3n}{4}-x|$, and $k \ne l$. A vertex $k$ from $v_{\left\lceil \frac{n}{4}\right\rceil }$ has a shortest path to $V(\text{I})$ that uses the edge $v_{\frac{n}{4}}v_{\frac{3n}{4}}$ at the $k+1$ position of the path to $V(\text{I})$. A vertex $l$ from $v_{\frac{3n}{4}}$ has a shortest path to$V(\text{I})$ that uses the edge $v_{\frac{n}{4}}v_{\frac{3n}{4}}$ at the $l+1$ position of the path to $V(\text{I})$. Since $k \ne l$, the vertices being compared will use the edge $v_{\frac{n}{4}}v_{\frac{3n}{4}}$ in their path to $V(\text{I})$ at different times. Since the edge $v_{\frac{n}{4}}v_{\frac{3n}{4}}$ is unique, and since the vertices in $A \cup B$ use the edge $v_{\frac{n}{4}}v_{\frac{3n}{4}}$ to get to $V(\text{I})$ at a different time than the vertices in $C \cup B$, the vertices in $A \cup B$ that are $k$ away from $v_{\frac{n}{4}}$ are unique from vertices in $C \cup D$ that are $l$ away from  $v_{\frac{3n}{4}}$. 

In this way, vertices in $A \cup B$ are unique from vertices in set $C \cup D$.
\end{itemize}

\item[$\circ$] Comparing vertices in Set A to Set B: 

The edge $v_4v_{\frac{n}{2}-4}$ is unique because $v_4$ and $v_{\frac{n}{2}-4}$ are unique vertices: $v_4$ is always distance four from $V(\text{III})$ and is always distance three from $V(\text{IV})$, and $v_{\frac{n}{2}-4}$ is always distance two from $V(\text{I})$ and is always distance two from $V(\text{IV})$. We compare the vertices in Set A to the vertices in Set B by comparing the path from $v_x$ to the vertex $v_{\frac{n}{2}-4}$, that uses the edge $v_4v_{\frac{n}{2}-4}$. We find that the path from any vertex $v_x$ in Set A to the vertex $v_{\frac{n}{2}-4}$, that uses the edge $v_4v_{\frac{n}{2}-4}$, is always $i$ edges short, where $i=x-3$, from a $2i$ cycle. Whereas, the path from any vertex $v_x$ in Set B to the vertex $v_{\frac{n}{2}-4}$, that uses the edge $v_4v_{\frac{n}{2}-4}$, is always $j$ edges short, where $j=\frac{n}{2}-4-x$, from a $2j+2$ cycle. When comparing vertices in Set A to vertices in Set B we assign an ordered pair $(r, s)$ to each $v_x$ in each set, where $s$ is the number of edges that $v_x$ is short of an $r$-cycle. In Set A each $v_x$ is assigned the ordered pair $(2i, i)$ and in Set B each $v_x$ is assigned the ordered pair $(2j+2, j)$. We will show that a vertex in Set A will never have the same ordered pair as a vertex in set B. There are two cases: 

\begin{itemize}
\item Case 1: $i=j$

When $i=j$, then $(2i, i)$ and $(2j+2,j)$  will never be the same, as if $i=j$, we know $2i \ne 2j+2$. Thus, the $(2i, i) \ne (2j+2,j)$, and the vertices in Set A are unique from the vertices in Set B. 

\item Case 2: $i \ne j$

If $i \ne j$, then $(2i, i)$ and $(2j+2,j)$ are different.  In this way the vertices in Set A are always unique from the vertices in Set B. 

\end{itemize}

In this way the vertices in Set A are always unique from the vertices in Set B.

\item[$\circ$] Comparing vertices within Set A:

We compare a vertex $v_x$ in Set A to all other vertices within Set A by studying the path from $v_x$ to the vertex $v_{\frac{n}{2}-4}$, that uses the edge $v_4v_{\frac{n}{2}-4}$. This path will always be $i$ edges short, where $i=x-3$, from a $2i$ cycle. Every $v_x$ in Set A has a unique $i$ value because $i=x-3$ and every vertex has a unique $x$ value assigned. Thus, every vertex in Set A is unique, as every vertex is short of a cycle by a unique number of edges. In this way, every vertex in Set A is unique. 

\item[$\circ$] Comparing vertices within Set B:

We compare a vertex $v_x$ in Set B to all other vertices within Set B by studying the path from $v_x$ to the vertex $v_{\frac{n}{2}-4}$, that uses the edge $v_4v_{\frac{n}{2}-4}$. This path will always be $j$ edges short of a cycle, where $j=\frac{n}{2}-4-x$, from a $2j+2$ cycle. Every $v_x$ in Set B has a unique $j$ value because $j=\frac{n}{2}-4-x$ and every vertex has a uniquet $x$ value assigned. Thus, every vertex in Set B is unique, as every vertex is short of a cycle by a unique number of edges. In this way, every vertex in Set B is unique. 

\item[$\circ$] Comparing vertices in Set D to Set C:

The edge $v_{n-3}v_{\frac{n}{2}+3}$ is unique because $v_{n-3}$ and $v_{\frac{n}{2}+3}$ are unique vertices: $v_{n-3}$ is in $V(\text{IV})$ and distance three from $n$, and $v_{\frac{n}{2}+3}$ is always distance one from $V(\text{IV})$ and is always distance three from $V(\text{II})$. We compare the vertices in Set D to the vertices in Set C by comparing the path from $v_x$ to the vertex $v_{\frac{n}{2}+3}$ that uses the edge $v_{n-3}v_{\frac{n}{2}+3}$. We find that the path from any vertex $v_x$ in Set D to the vertex $v_{\frac{n}{2}+3}$, that uses the edge $v_{n-3}v_{\frac{n}{2}+3}$, is always $m$ edges short, where $m=n-x-2$, from a $2m$ cycle. Whereas, the path from any vertex $v_x$ in Set C to the vertex $v_{\frac{n}{2}+3}$, that uses the edge $v_{n-3}v_{\frac{n}{2}+3}$, is always $q$ edges short, where $q=x-\frac{n}{2}+3$, from a $2q+2$ cycle. When comparing vertices in Set D to vertices in Set C we assign an ordered pair $(r, s)$ to each $v_x$ in each set, where $s$ is the number of edges that $v_x$ is short of an $r$-cycle. In Set C each $v_x$ is assigned the ordered pair $(2q+2, q)$ and in Set D each $v_x$ is assigned the ordered pair $(2m, m)$. We will show that a vertex in Set D will never have the same ordered pair as a vertex in set C. There are two cases: 

\begin{itemize}
\item Case 1: $m=q$

When $m=q$, then $(2m, m)$ and $(2q+2,q)$  will never be the same, as if $m=q$, we know $2m \ne 2q+2$. Thus, the $(2m, m) \ne (2q+2,q)$, and the vertices in Set D are unique from the vertices in Set C. 

\item Case 2: $m \ne q$

If $m \ne q$, then $(2m, m)$ and $(2q+2,q)$ are different.  In this way the vertices in Set D are always unique from the vertices in Set C. 

\end{itemize}

\noindent In this way the vertices in Set D are always unique from the vertices in Set C.

\item[$\circ$] Comparing vertices within Set C:

We compare a vertex $v_x$ in Set C to all other vertices within Set C by studying the path from $v_x$ to the vertex $v_{\frac{n}{2}-4}$, that uses the edge $v_{n-3} v_{\frac{n}{2}+3}$. This path will always be $m$ edges short, where $m=n-x-2$, from a $2m$ cycle. Every $v_x$ in Set C has a unique $m$ value because  $m=n-x-2$ and every vertex has a unique $x$ value assigned. Thus, every vertex in Set C is unique, as every vertex is short of a cycle by a unique number of edges. In this way, every vertex in Set C is unique. 

\item[$\circ$] Comparing vertices within Set D:

We compare a vertex $v_x$ in Set D to all other vertices within Set D by studying the path from $v_x$ to the vertex $v_{\frac{n}{2}-4}$, that uses the edge $v_{n-3} v_{\frac{n}{2}+3}$. This path will always be $q$ edges short, where $q=x-\frac{n}{2}+3$, from a $2q+2$ cycle. Every $v_x$ in Set D has a unique $q$ value because $q=x-\frac{n}{2}+3$ and every vertex has a unique $x$ value assigned. Thus, every vertex in Set D is unique, as every vertex is short of a cycle by a unique number of edges. In this way, every vertex in Set D is unique.

\noindent In this way, any vertex in any set A, B, C, and D, is unique from all other vertices. 

\end{itemize}

\noindent In this way, all vertices are unique in this construction of a quartic Hamiltonian graph on any $n \geq 20 $, $n \equiv 0  \; \text{mod} \; 4$. Therefore, the graph is asymmetric. 
\end{itemize}


\subsubsection{When $n \equiv 2 \; \text{mod} \; 4$}

\textbf{Vertices that are always contained in a 3-Cycle}

\vspace{-3pt} \noindent Vertices $v_1, \; v_2, \; v_{\frac{n}{2} - 2}, \; v_{\frac{n}{2} - 1}, \; v_{\frac{n}{2}}, \; v_{\frac{n}{2} + 1}, \; v_{\frac{n}{2} + 2}, \; v_{n-3}, \; v_{n-2}, \; v_{n-1}, \; v_n$ are the only vertices contained in any $C_3$ subgraph.

\begin{itemize}
\item \vspace{-5pt} The vertex $v_{\frac{n}{2}}$ is unique because it is the only vertex contained in two distinct $C_3$s.

\item \vspace{-5pt} The vertices $v_2$ and $v_{n-3}$ are the only vertices contained in a $C_3$ that is not distance one away from another $C_3$.

\begin{itemize}
\vspace{-5pt}
\item $v_2$ is distance three away from the vertex $v_{\frac{n}{2}}$, while the vertex $v_{n-3}$ is not. 
\end{itemize}

\item \vspace{-5pt} The vertices $v_1, \; v_{\frac{n}{2} - 1}, \; v_{\frac{n}{2} + 1}, \; v_{n-2}, \; v_{n-1}, \; v_{\frac{n}{2}-2}, \; v_{\frac{n}{2}+2} \; \text{and} \; v_n$ are all contained in a $C_3$ and are one away from a $C_3$.


\begin{itemize}
\item \vspace{-5pt} $v_1$ and $v_n$ are adjacent to $v_2$.

\begin{itemize}
\item \vspace{-3pt}$v_1$ is only contained in one $C_4$ whereas $v_n$ is contained in two $C_4$s. 
\end{itemize}
\end{itemize}
\vspace{-7pt}

\begin{itemize}
\item \vspace{-5pt}$v_{\frac{n}{2} - 1}$, $v_{\frac{n}{2} + 1}$, $v_{\frac{n}{2}-2}$, and \;$ v_{\frac{n}{2}+2}$ are adjacent to the vertex $v_{\frac{n}{2}}$.

\begin{itemize}
\item \vspace{-1pt} $v_{\frac{n}{2} - 1}$ is adjacent to $v_n$, while $v_{\frac{n}{2} + 1}$, $v_{\frac{n}{2}-2}$, and \;$ v_{\frac{n}{2}+2}$ are not.

\item \vspace{-1pt} $v_{\frac{n}{2} + 1}$ is adjacent to $v_1$, while $v_{\frac{n}{2} - 1}$, $v_{\frac{n}{2}-2}$, and \;$ v_{\frac{n}{2}+2}$ are not.

\item \vspace{-1pt} $v_{\frac{n}{2} - 2}$ is distance two away from $v_n$, while $v_{\frac{n}{2} + 2}$ is not.
\end{itemize}

\item \vspace{-3pt} $v_{n-1}$ and $v_{n-2}$ are not adjacent to $v_{\frac{n}{2}}$ nor $v_2$.

\begin{itemize}
\item \vspace{-3pt} $v_{n-1}$ is adjacent to $v_n$, while the vertex $v_{n-2}$ is not.
\end{itemize}

\end{itemize}
\end{itemize}

\noindent \textbf{Vertices That are not contained in a 3-cycle}
\newline Let any vertex $v_k$ such that  $n-6 \leq k \leq 6$, $\frac{n}{2}-5 \leq k \leq \frac{n}{2}+4$ be an element of $\mathcal{U}$.

In order to show that the vertices, $6<x<\frac{n}{2}-5$ and $\frac{n}{2}+4<x<n-6$, are unique, we begin with a vertex $v_x$ and check to see if its adjacent vertices are elements of $\mathcal{U}$. If all of the vertices adjacent to $v_x$ are unique, we can say that the vertex is unique and conclude it belongs to $\mathcal{U}$. 

We begin with the vertex $v_{6+t}$, where $t=0$ and visit the vertices adjacent to $v_{6+t}$, that is: $v_{5+t}$, $v_{7+t}$, $v_{n-8-t}$, and $v_{n-4-t}$. We check if the adjacent vertices are elements of $\mathcal{U}$. If at least two of these vertices belong to $\mathcal{U}$, then compare the distance from each of the vertices to $v_{6+t-3}$. We find that the shortest path from $v_{7+t}$ to $v_{6+t-3}$ is always two, and the shortest path from $v_{n-8-t}$ to $v_{6+t-3}$ is always four. Thus, the vertices are different, as they have different distances to $v_{6+t-3}$. We can now conclude the vertices $v_{5+t}$, $v_{7+t}$, $v_{n-8-t}$, and $v_{n-4-t}$ belong to $\mathcal{U}$. 

If there are fewer than two vertices belonging to $\mathcal{U}$, then check the distance from each vertex to $v_{6+t-3}$. We find that the shortest paths from the vertices $v_{7+t}$, $v_{n-8-t}$, and $v_{n-4-t}$ to the vertex $v_{6+t-3}$ are, two, four, and two, respectively. Since two vertices have the same distance to $v_{6+t-3}$, we compare their distances to $v_{6+t-4}$. Since the shortest path distance from $v_{7+t}$ to $v_{6+t-4}$ is always three, and since the shortest path from $v_{n-4-t}$ to $v_{6+t-4}$ is always one, the vertices are unique, as each vertex has a unique distance to both $v_{6+t-3}$ and $v_{6+t-4}$. We can now say the vertices $v_{5+t}$, $v_{7+t}$, $v_{n-8-t}$, and $v_{n-4-t}$ belong to $\mathcal{U}$.

We then add one to $t$ and check the adjacent vertices as described, until we reach the vertex $v_{\frac{n}{2}-6}$. We know the vertex $v_{\frac{n}{2}-6}$ is unique because all of its adjacent vertices are elements of $\mathcal{U}$. At this point, the cardinality of $\mathcal{U}$ is equal to the number of vertices in the graph, which implies the graph is asymmetric. 


\section{Conclusion}

\hspace{5pt} We constructed and presented an infinite family of cubic and quartic Hamiltonian graphs of even orders, starting at $n=12$. We note that the complement of an asymmetric $3$-regular Hamiltonian graph on $n$ vertices is an asymmetric $(n-4)$-regular graph. An application of Dirac's Theorem shows that this graph is Hamiltonian. This covers all even $n\geq 8$. The complement of an asymmetric $4$-regular Hamiltonian graph on $n$ vertices is an asymmetric $(n-5)$-regular graph. An application of Dirac's Theorem shows that this graph is Hamiltonian. This covers all odd $n\geq 7$. 

In Figure 6 we present an asymmetric $5$-regular Hamiltonian graph with 12 vertices.
\begin{figure}[h!]
\begin{center}
\includegraphics[scale=0.3]{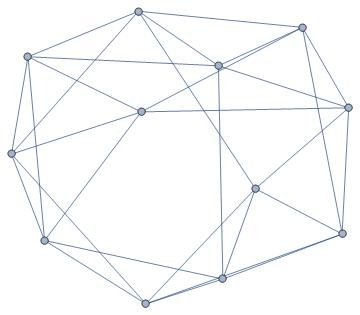}
\end{center}
\caption{An asymmetric $5$-regular Hamiltonian graph with 12 vertices}
\end{figure}
The complement of this graph will be an asymmetric $6$-regular graph on 12 vertices. It follows by Dirac's Theorem that this graph is Hamiltonian. 

\hspace{5pt} We established the existence of infinite families of $k$-regular asymmetric Hamiltonian graphs for $k=3$ and $k=4$. It would be interesting to determine for which $k>4$ there exists an infinite family of asymmetric $k$-regular Hamiltonian graphs.


\section*{Acknowledgements}
This research was supported by the National Science Foundation Research Experiences for Undergraduates Award 1659075. Part of this work originated in an undergraduate project at Whitman College by Luke Rodriguez written under the direction of Professor Barry Balof.

\bibliographystyle{plain}
\bibliography{citations}
\end{document}